\documentclass[11pt]{amsart}

\usepackage[margin=1.1in]{geometry}
\usepackage{amsmath,amssymb,amsthm,mathtools}
\usepackage{microtype}
\usepackage[hidelinks]{hyperref}

\newtheorem{theorem}{Theorem}
\newtheorem{lemma}[theorem]{Lemma}

\theoremstyle{remark}

\newtheorem{question}[theorem]{\rm\bf Question}

\DeclareMathOperator{\Area}{Area}
\DeclareMathOperator{\sys}{sys}
\DeclareMathOperator{\arcosh}{arcosh}

\title{Every Closed Oriented Surface of Positive Genus Is Loewner}
\author{Qiongling Li}
\address{Qiongling Li: Chern Institute of Mathematics and LPMC, Nankai University, Tianjin 300071, China} \email{qiongling.li@nankai.edu.cn}
\date{September 2, 2026}

\begin{document}
\maketitle

\begin{abstract}
Using the universal-cover ball comparison theorem of Chodosh--Vikman, we have an easy consequence that every closed oriented Riemannian surface of positive genus is Loewner, that is, it satisfies a systolic-area inequality. 

\end{abstract}
\subsection*{AI usage statement} The key observation in this note was suggested by ChatGPT
(GPT-5.6 Pro), which identified the recent preprint of
Chodosh--Vikman \cite{CV} and observed that with elementary calculus, it yields the main
result. ChatGPT also produced a substantial portion of the initial
draft. The author has checked and revised every argument in the
final version and accepts full responsibility for the content and
for any remaining errors.
\subsection*{Purpose of the note} This note is only for sharing and giving a closure to the question, which the author has been followed up for almost three years. The author does not intend to submit this paper for publication in any journal. 
\section{Introduction}
Throughout, all surfaces are connected, closed, and oriented.  Let $\Sigma$ be a Riemannian surface of genus $k\geq 1$. Denote the \emph{systole} of $\Sigma$ by $\sys(\Sigma)$. By definition, $\sys(\Sigma)$ is the length of shortest non-contractible closed geodesics. Denote the area of $\Sigma$ by 
$\Area(\Sigma)$. We say that the surface $\Sigma$ is \emph{Loewner} if the systolic ratio satisfies the inequality
\begin{equation}\label{equ:Loewner}
    \frac{\sys^2(\Sigma)}{\Area(\Sigma)} \leq \frac{2}{\sqrt{3}}.
\end{equation}

Around 1949, 
Loewner (ref. \cite{Pu}) first proved that every surface of genus $1$ satisfies \eqref{equ:Loewner}. 

For details on history of Loewner's inequality and related developments in  the systolic geometry,
we refer to the book of Katz \cite{katz2007}. It is natural to ask if every closed surface is Loewner. 

\begin{question}
Is every closed oriented Riemannian surface of positive genus
Loewner?
\end{question}

Gromov \cite{gromov1983} showed that, for any surface of genus $k$, 
\begin{equation*}\label{equ:Loewner}
    \frac{\sys^2(\Sigma)}{\Area(\Sigma)}< \frac{64}{4\sqrt{k}+27}.
\end{equation*}
Thus all surfaces of genus $k\geq 51$ is Loewner.  For surfaces of genus $2$ and genus $k\geq 20$, the answer is affirmative, shown by Katz and Sabourau \cite{ks2005,ks2006}. The author and Su \cite{LS} extends the result to $k\geq 18$.

We show the following theorem. 

\begin{theorem}\label{cor:all}
Every closed oriented Riemannian surface of genus at least one is Loewner.
\end{theorem}

\section{Proof of Theorem \ref{cor:all}}

We use the following theorem from \cite[Theorem 4]{CV}.

\begin{theorem}[Chodosh--Vikman]\label{thm:CV}
Let $(\Sigma,g)$ be a closed oriented Riemannian surface of genus $k\geq 2$, normalized by
\[
  \Area_g(\Sigma)=4\pi(k-1).
\]
Let $(\widetilde\Sigma,\widetilde g)$ be its universal Riemannian cover. Then, for every $r>0$,
\[
  \max_{\widetilde x\in\widetilde\Sigma}
  \Area_{\widetilde g}\bigl(B_{\widetilde g}(\widetilde x,r)\bigr)
  \geq 2\pi(\cosh r-1).
\]
\end{theorem}

\begin{lemma}\label{lem:embed}
Let $(\Sigma,g)$ be a closed Riemannian manifold, let
\[
  p:(\widetilde\Sigma,\widetilde g)\longrightarrow(\Sigma,g)
\]
be its universal Riemannian covering, and set $s=\sys_g(\Sigma)$. If $0<r<s/2$, then the restriction of $p$ to every ball
\[
  B_{\widetilde g}(\widetilde x,r)
\]
is injective. Consequently,
\[
  \Area_{\widetilde g}\bigl(B_{\widetilde g}(\widetilde x,r)\bigr)
  \leq \Area_g(\Sigma).
\]
\end{lemma}

\begin{proof}
Suppose that distinct points $y,z\in B_{\widetilde g}(\widetilde x,r)$ have the same image under $p$. Then $z=\gamma y$ for a nontrivial deck transformation $\gamma$. Joining $y$ to $z$ by first going from $y$ to $\widetilde x$ and then from $\widetilde x$ to $z$ gives a path of length strictly less than $2r<s$. Its projection is a closed curve representing the nontrivial element $\gamma\in\pi_1(\Sigma)$, contradicting the definition of the systole. Thus $p$ is injective on the ball. Since $p$ is a local isometry, it preserves area on this ball, and its image is contained in $\Sigma$.
\end{proof}

\begin{theorem}\label{thm:strong}
Let $(\Sigma,g)$ be a closed oriented Riemannian surface of genus $k\geq2$. Then
\begin{equation}\label{eq:strong-bound}
  \frac{\sys_g(\Sigma)^2}{\Area_g(\Sigma)}
  \leq
  \frac{\arcosh^2(2k-1)}{\pi(k-1)}
  \leq
  \frac{\arcosh^2 3}{\pi}
  <\frac{2}{\sqrt3}.
\end{equation}
In particular, every such surface is Loewner.
\end{theorem}

\begin{proof}
Write
\[
  A=\Area_g(\Sigma),
  \qquad
  s=\sys_g(\Sigma),
\]
and rescale the metric by
\[
  \widehat g=\lambda^2g,
  \qquad
  \lambda^2=\frac{4\pi(k-1)}{A}.
\]
Then
\[
  \Area_{\widehat g}(\Sigma)=4\pi(k-1),
  \qquad
  \widehat s:=\sys_{\widehat g}(\Sigma)=\lambda s,
\]
and the systolic ratio is scale invariant:
\[
  \frac{s^2}{A}
  =
  \frac{\widehat s^2}{4\pi(k-1)}.
\]

Fix $0<r<\widehat s/2$. By Theorem~\ref{thm:CV}, there is a point $\widetilde x$ in the universal cover such that
\[
  \Area_{\widetilde{\widehat g}}
  \bigl(B_{\widetilde{\widehat g}}(\widetilde x,r)\bigr)
  \geq 2\pi(\cosh r-1).
\]
By Lemma~\ref{lem:embed}, this ball embeds isometrically into $(\Sigma,\widehat g)$. Hence
\[
  4\pi(k-1)
  =\Area_{\widehat g}(\Sigma)
  \geq 2\pi(\cosh r-1).
\]
Letting $r\nearrow\widehat s/2$ gives
\[
  \cosh\!\left(\frac{\widehat s}{2}\right)\leq 2k-1,
\]
and therefore
\[
  \widehat s\leq 2\arcosh(2k-1).
\]
It follows that
\begin{equation}\label{eq:first-bound}
  \frac{s^2}{A}
  =\frac{\widehat s^2}{4\pi(k-1)}
  \leq
  \frac{\arcosh^2(2k-1)}{\pi(k-1)}.
\end{equation}

It remains to compare the right-hand side with the Loewner constant. For real $x>1$, put
\[
  F(x)=\frac{\arcosh^2(2x-1)}{\pi(x-1)}.
\]
Writing $t=\arcosh(2x-1)$, so that $x-1=(\cosh t-1)/2$, we have
\[
  F(x)=\frac{2t^2}{\pi(\cosh t-1)}.
\]
Moreover,
\[
  \frac{d}{dt}\left(\frac{2t^2}{\pi(\cosh t-1)}\right)
  =
  \frac{2t\bigl(2(\cosh t-1)-t\sinh t\bigr)}{\pi(\cosh t-1)^2}
  \leq0,
\]
because
\[
  2(\cosh t-1)\leq t\sinh t
  \quad\Longleftrightarrow\quad
  2\tanh\!\left(\frac t2\right)\leq t,
\]
and $\tanh u\leq u$ for $u\geq0$. Thus $F$ is decreasing, and for every integer $k\geq2$,
\[
  \frac{\arcosh^2(2k-1)}{\pi(k-1)}\leq \frac{\arcosh^2 3}{\pi}<\frac87<\frac{2}{\sqrt{3}}.
\]

Combining this with \eqref{eq:first-bound} proves \eqref{eq:strong-bound}.
\end{proof}

\begin{proof}(of Theorem \ref{cor:all})
For genus one this is Loewner's classical torus theorem. For genus at least two it follows from Theorem~\ref{thm:strong}.
\end{proof}

\bibliographystyle{siam}

\end{document}